\documentclass{amsart}
\usepackage{amsmath, amssymb,epic,graphicx,mathrsfs,enumerate}
\usepackage[all]{xy}

\usepackage{amsthm}
\usepackage{amssymb}
\usepackage{latexsym}
\usepackage{longtable}
\usepackage{epsfig}
\usepackage{amsmath}
\usepackage{hhline}
\usepackage{amscd}
\usepackage{newlfont}

 \DeclareMathOperator{\perm}{Sym}
 \DeclareMathOperator{\soc}{soc}
\DeclareMathOperator{\aut}{Aut} \DeclareMathOperator{\out}{Out}

 \DeclareMathOperator{\frat}{Frat}

 \DeclareMathOperator{\AGL}{AGL}

\DeclareMathOperator{\core}{Core}

\DeclareMathOperator{\alt}{Alt}
\DeclareMathOperator{\End}{End} \DeclareMathOperator{\h}{{H^1}}
\DeclareMathOperator{\der}{Der}
\DeclareMathOperator{\ider}{InnDer} \DeclareMathOperator{\diag}{diag}

\newtheorem{thm}{Theorem}

 \newtheorem{lemma}[thm]{Lemma}
\newtheorem{prop}[thm]{Proposition}

\numberwithin{equation}{section}

\renewcommand{\footnote}{\endnote}
\newcommand{\ignore}[1]{}\makeglossary
\begin{document}
\bibliographystyle{amsplain}


\title[The Chebotarev invariant]{The Chebotarev invariant of a finite  group:\\ a conjecture of  Kowalski and  Zywina}

\author{Andrea Lucchini}
\address{Dipartimento di Matematica,
	Via Trieste 63, 35121 Padova, Italy.}
\email{lucchini@math.unipd.it}
\thanks{Partially supported by Universit\`a di Padova (Progetto di Ricerca di Ateneo: \lq\lq Invariable generation of groups\rq\rq).}

\begin{abstract}A subset $\{g_1, \ldots , g_d\}$ of a finite group $G$  invariably generates $G$ if
	$\{g_1^{x_1}, \ldots , g_d^{x_d}\}$ generates $G$ for every choice of $x_i \in G$. The Chebotarev invariant $C(G)$ of $G$ is the expected value of the random variable $n$ that is minimal
	subject to the requirement that $n$ randomly chosen elements of $G$ invariably generate $G$. Confirming a conjecture of Kowalski and  Zywina, we prove
	that there exists an absolute constant $\beta$ such that $C(G) \leq  \beta\sqrt{|G|}$
	for all finite  groups $G.$
\end{abstract}
\maketitle

\section{Introduction}
We say that
a subset $\{g_1, \ldots , g_d\}$ of a finite group $G$  invariably generates $G$ if
$\{g_1^{x_1}, \ldots , g_d^{x_d}\}$ generates $G$ for every choice of $x_i \in G$. The Chebotarev invariant $C(G)$ of $G$ is the expected value of the random variable $n$ that is minimal
subject to the requirement that $n$ randomly chosen elements of $G$ invariably generate $G$.
The main motivation for introducing the invariant $C(G)$ is the relationship to Chebotarev's Theorem and the calculation of Galois groups of polynomials with integer coefficients.
Chebotarev's Theorem
 provides elements of a suitable Galois group $G,$ where the elements are obtained only up to
conjugacy in $G;$  the interest in the study of $C(G)$ comes from computational group theory, where there is a need to
know how long one should expect to wait in order to ensure that choices of representatives from the
conjugacy classes provided by Chebotarev's Theorem will generate $G.$ This is discussed more carefully
in \cite{dixon} and \cite{kz}.

\

 In response to a question of Kowalski
and Zywina \cite{kz}, Kantor, Lubotzky and Shalev \cite{ig} bounded the size of a randomly
chosen set of elements of $G$ that is likely to generate $G$ invariably. As a corollary of their result, they proved
that there exists an absolute constant $c$ such that
$C(G) \leq c\sqrt{|G|\log|G|}$
for all finite groups $G$ (\cite[Theorem 1.2]{ig}).
This bound is close to best possible: as it is noticed in \cite{ig},  sharply 2-transitive groups provide an infinite family of groups $G$ for which $C(G) \thicksim  \sqrt{|G|}$. In particular, $C(\AGL(1,q)) \thicksim q$ as
$q \to \infty$ \cite[Proposition 4.1]{kz}. In fact \cite[Section 9]{kz}
asks whether $C(G) = O(\sqrt{|G|})$ for all finite groups $G.$ In this paper we give
an affirmative answer.

\begin{thm}\label{mainuno}
	There exists an absolute constant $\beta$ such that $C(G) \leq  \beta\sqrt{|G|}$
	for all finite  groups $G.$
\end{thm}

For $k\geq 1,$ let $P_I(G,k)$ be the probability that $k$ randomly chosen elements of $G$ generate $G$ invariably.
An easy argument in probability theory shows that if $P_I(G,k)\geq \epsilon,$ then $C(G)\leq k/\epsilon.$ Indeed we obtain
Theorem \ref{mainuno} as a corollary of the following result.

\begin{thm}\label{due}
	For any $\epsilon > 0$ there exists $\tau_\epsilon$ such that $P_I(G,k)\geq 1-\epsilon$ for any finite  group $G$ and any $k\geq  \tau_\epsilon\sqrt{|G|}.$
\end{thm}

One of the ingredients used in the proof of Theorem \ref{due} is the notion of crown, introduced by Gasch\"utz in \cite{gaz} in the case of finite solvable groups and
 generalized in \cite{paz} to arbitrary finite groups. The property of the crowns are enough to prove the theorem in the case of solvable groups, but in order to apply our
 arguments to arbitrary finite groups, we need some results
relying on the classification of the finite simple groups.
The first is a bound on the order of the first cohomology group
of a finite group over a faithful irreducible module: if $V$ is an irreducible faithful $G$-module over
a finite field, then $|\h(G,V)|\leq {\sqrt{V}}<|V|$ (see \cite {AG}
and \cite{guho}). This result is near to be sufficient for our purposes, but we need
a more precise information in the particular case when $|V|\leq |G|$ and the proportion of elements of $G$ fixing no nontrival vector of $V$ is small (see Proposition \ref{diff}). Other two consequences of the classification of the finite simple groups are necessary to prove Lemma \ref{nonabcr}: there exists an absolute constant
$c_1$ such that any finite group $G$ has at most $c_1|G|^{3/2}$
maximal subgroups \cite[Theorem 1.3]{lps}; the proportion of fixed-point-free permutations
in a non-affine primitive group of degree $n$ is at least $c_2/ \log n,$ for some absolute constant
$c_2 > 0$ \cite[Theorem 8.1]{fg1}.
This last result in turn relies on a conjecture
made independently by Boston  and Shalev, stating that there exists an absolute constant $\epsilon > 0$ such that the proportion of fixed-point-free elements in any finite simple
transitive permutation group is at least $\epsilon$. This conjecture was proved for alternating groups by {\L}uczak and L. Pyber in \cite{lp} and for the simple groups of Lie type by Fulman and Guralnick in a series of four papers (\cite{fg1}, \cite{fg}, \cite{fgint}, \cite{fglast}).

\section{Crowns in finite groups}

Let $L$ be   a monolithic primitive group
and let $A$ be its unique minimal normal subgroup. For each positive integer $k$,
let $L^k$ be the $k$-fold direct
product of $L$.
The crown-based power of $L$ of size  $k$ is the subgroup $L_k$ of $L^k$ defined by
$$L_k=\{(l_1, \ldots , l_k) \in L^k  \mid l_1 \equiv \cdots \equiv l_k \ {\mbox{mod}} A \}.$$
Equivalently, $L_k=A^k \diag L^k$. 

Following  \cite{paz},  we say that
two irreducible $G$-groups $A$ and $B$  are  {$G$-equivalent} and we put $A \sim_G B$, if there is an
isomorphism $\Phi: A\rtimes G \rightarrow B\rtimes G$ such that the following diagram commutes:

\begin{equation*}
\begin{CD}
1@>>>A@>>>A\rtimes G@>>>G@>>>1\\
@. @VV{\phi}V @VV{\Phi}V @|\\
1@>>>B@>>>B\rtimes G@>>>G@>>>1
\end{CD}
\end{equation*}

\

Note that two $G$\nobreakdash-isomorphic
$G$\nobreakdash-groups are $G$\nobreakdash-equivalent. In the particular case where   $A$ and $B$ are abelian the converse is true:
if $A$ and $B$ are abelian and $G$\nobreakdash-equivalent, then $A$
and $B$ are also $G$\nobreakdash-isomorphic.
It is proved (see for example \cite[Proposition 1.4]{paz}) that two  chief factors $A$ and $B$ of $G$ are  $G$-equivalent if and only if  either they are  $G$-isomorphic between them or there exists a maximal subgroup $M$ of $G$ such that $G/\core_G(M)$ has two minimal normal subgroups $N_1$ and $N_2$
$G$-isomorphic to $A$ and $B$ respectively. For example, the  minimal normal subgroups of a crown-based power $L_k$ are all $L_k$-equivalent.

Let $A=X/Y$ be a chief factor of $G$. A complement $U$ to $A$ in $G$ is a subgroup $U $ of $G$ such that $UX=G$ and $U \cap X=Y$. We say that   $A=X/Y$ is a Frattini chief factor if  $X/Y$ is contained in the Frattini subgroup of $G/Y$; this is equivalent to say that $A$ is abelian and there is no complement to $A$ in $G$.
The  number $\delta_G(A)$  of non-Frattini chief factors $G$-equivalent to $A$   in any chief series of $G$  does not depend on the series. Now,
we denote by  $L_A$  the  monolithic primitive group  associated to $A$,
that is
$$L_{A}=
\begin{cases}
A\rtimes (G/C_G(A)) & \text{ if $A$ is abelian}, \\
G/C_G(A)& \text{ otherwise}.
\end{cases}
$$
If $A$ is a non-Frattini chief factor of $G$, then $L_A$ is a homomorphic image of $G$.
More precisely,  there exists
a normal subgroup $N$ of $G$ such that $G/N \cong L_A$ and $\soc(G/N)\sim_GA$. Consider now  all the normal subgroups $N$ of $G$ with the property that  $G/N \cong L_A$ and $\soc(G/N)\sim_G A$:
the intersection $R_G(A)$ of all these subgroups has the property that  $G/R_G(A)$ is isomorphic to the crown-based  power $(L_A)_{\delta_G(A)}$.
The socle $I_G(A)/R_G(A)$ of $G/R_G(A)$ is called the $A$-crown of $G$ and it is  a direct product of $\delta_G(A)$ minimal normal subgroups $G$-equivalent to $A$.

\begin{lemma}{\cite[Lemma 1.3.6]{classes}}\label{corona}
		Let $G$ be a finite  group with trivial Frattini subgroup. There exists
		a crown $I_G(A)/R_G(A)$ and a non trivial normal subgroup $U$ of $G$ such that $I_G(A)=R_G(A)\times U.$
	\end{lemma}
	
	\begin{lemma}{\cite[Proposition 11]{crowns}}\label{sotto} Assume that $G$ is a finite  group with trivial Frattini subgroup and let $I_G(A), R_G(A), U$ be as in the statement of Lemma \ref{corona}. If $KU=KR_G(A)=G,$ then $K=G.$
	\end{lemma}

\section{Crown-based powers with abelian socle}

In this section we will assume that $H$ is a  finite group acting irreducibly and faithfully on an elementary abelian  $p$-group $V$. The semidirect product $L=V\rtimes H$ is a monolithic primitive group. 
For a positive integer $u$ we consider the crown-based power $L_u$: we have that $L_u$ is isomorphic to the semidirect product $G = V^u \rtimes H$, where we assume that  the action of $H$ is diagonal on $V^u$, that is, $H$ acts in the
same way on each of the $u$ direct factors. We assume that
$h_1,\dots,h_d$ (invariably) generate $H$ and we look for conditions ensuring the existence of $d$-elements $w_1,\dots,w_d\in V^u$ such that $h_1w_1,\dots,h_dw_d$ (invariably) generate $G.$ The case when $H=1$ is trivial:
$V\cong C_p$ is a cyclic group of  prime order  and $G=C_p^u$ can be generated by $d$ elements $w_1,\dots,w_d$ if and only if $u\leq d.$ So for the remaining
part of this section we will assume $H\neq 1.$
We will denote by $\der(H,V)$ the set of the derivations from
$H$ to $V$ (i.e. the maps $\delta: H\to V$ with the property that $\delta(h_1h_2)=\delta(h_1)^{h_2}+\delta(h_2)$ for every $h_1,h_2\in H$). If $v\in V$ then the map $\delta_v:H\to V$
defined by $\delta_v(h)=[h,v]$ is a derivation. The set  $\ider(H,V)=\{\delta_v\mid v \in V\}$ of the inner derivations from $H$ to $V$ is a subgroup of $\der(V,H)$ and the factor group $\h(H,V)=\der(H,V)/\ider(H,V)$ is the first cohomology group of $H$
with coefficients in $V.$

The following is a generalization of a similar partial result (\cite[Proposition 2.1]{cl}),
proved in the particular case when $H$ is soluble, or, more in general, when 
$\h(H,V)=0.$
\begin{prop}\label{crit} Suppose that $H=\langle h_1, \dots, h_d\rangle$.
 Let $w_i=(w_{i,1},\dots,w_{i,u})\in V^u$ with $1\leq i\leq d.$ The following are equivalent.
\begin{enumerate}
\item $G \neq \langle h_1w_1,\dots,h_dw_d\rangle$;
\item there exist $\lambda_1,\dots,\lambda_u \in F=\mathrm{End}_{H}(V)$ and a derivation $\delta \in \der(H,V)$
with $(\lambda_1,\dots,\lambda_u, \delta) \not= (0, \ldots,0,0)$ such
that $\sum_{1\leq j\leq u} \lambda_j w_{i,j} = \delta(h_i)$ for each $i\in\{1,\dots,d\}.$
\end{enumerate}
\end{prop}
\begin{proof}
Let  $K=\langle h_1w_1,\dots,h_dw_d\rangle$. First we  prove, by induction on $u$, that if $K\neq G$
then $(2)$ holds. Let $z_i=h_i(w_{i,1},\dots,w_{i,u-1},0)$ and let  $Z=\langle z_1, \dots, z_d
\rangle$. If $Z \not\cong V^{u-1}H$, then, by induction,
there exist $\lambda_1,\dots,\lambda_{u-1} \in F$ and $\delta \in \der(H,V)$
with $(\lambda_{1}, \ldots , \lambda_{u-1},\delta) \not=
(0,\ldots,0,0)$ such that $\sum_{1\leq j\leq u-1} \lambda_j w_{i,j} = \delta(h_i)$ for each $i\in\{1,\dots,d\}.$ In this case
$\lambda_1,\dots,\lambda_{u-1},0$ and $\delta$ are the requested
elements.

 So we may assume $Z \cong V^{u-1}H$. Set
$V_u=\{(0,\dots,0,v)\mid v \in V\}$. We have $ZV_u = KV_u=G$
and $Z \neq G$; this implies that $Z$ is a complement of $V_u$ in
$G$ and therefore there exists $\delta^* \in \der(Z,V_u)$
such that $\delta^*(z_i)=w_{i,u}$ for each $i\in\{1,\dots,d\}.$ By Propositions 2.7 and 2.10 of \cite{AG},  there exist $\delta \in \der(H,V)$ and $\lambda_1,\dots,\lambda_{u-1}
\in F$ such that for each $h(v_1,\dots,v_{u-1},0) \in Z$ we
have
$$ \delta^*(h(v_1,\dots,v_{u-1},0))=\delta(h)+\lambda_1v_1 +
\dots+ \lambda_{u-1}v_{u-1}.$$
 In particular $-\sum_{1\leq j\leq u-1} \lambda_j w_{i,j}+w_{i,u}=\delta(h_i)$ for each $i\in\{1,\dots,d\},$
 hence  $(2)$ holds.

Conversely, if (2) holds then $\langle h(v_1,\dots,v_u)\mid
\delta(h)=\lambda_1v_1+\dots +\lambda_uv_u\rangle$ is a proper subgroup
of $G$ containing $K$. 
\end{proof}

Notice that $V,$ $\der(H,V)$
and $\h(H,V)$ are vector spaces over $F=\mathrm{End}_{H}(V)$. Let
$n:=\dim_FV=\dim_F\ider(H,V)$ and $m:=\dim_F\h(H,V).$ Clearly, we have
$\dim_F\der(H,V)=n+m.$ 

 Let $\pi_i:V^u \mapsto V$ be the canonical projection on the $i$-th component:
 $$\pi_i(v_1, \dots , v_u)=v_i.$$
Let  $w_i=(w_{i,1},\dots,w_{i,u})\in V^u$, for $i\in\{1, \dots ,d\}$, and consider the vectors
$$r_j= (\pi_j(w_1), \dots , \pi_j(w_d))=(w_{1,j}, \dots, w_{d,j})  \in V^d \text{ for }j\in\{1, \dots ,u\}.$$
Proposition \ref{crit} says that  the elements $h_1w_1,\dots,
h_dw_d$ generate a proper subgroup of $G$ if and only
if
there exists a non-zero vector $(\lambda_{1}, \ldots ,
\lambda_{u}, \delta)$ in $F^{u} \times \der(H,V)$ such that
$$
\sum_{1\leq j\leq u} \lambda_j r_j = \big(\delta(h_1), \dots , \delta(h_d)\big).
$$
Equivalently,  $\langle h_1w_1,\dots,h_dw_d \rangle=
G$ if and only if
$r_1, \dots , r_u$ in $V^d$  are linearly independent modulo the vector space
$$D=\{ \big(\delta(h_1), \dots , \delta(h_d)\big)\in V^d \mid w\in V\}.$$
Since $G=\langle h_1,\dots,h_d\rangle,$ the map $\der(H,V)\to D$
defined via $\delta\mapsto (\delta(h_1)\dotsm\delta(h_d))$
is an $F$-isomorphism. In particular $\dim_F(D)=\dim_F(\der(H,V))=n+m$ and so we conclude that there exist elements $w_1,\dots,w_d$
in $V^u$ such that $\langle h_1w_1,\dots,h_dw_d \rangle=
G$ if and only if
$u \leq \dim_F (V^d)-\dim_F(D)=n(d-1)-m$.

\

We now discuss the same question in the case of invariable generation, generalizing
to an arbitrary irreducible $H$-module $V$ a partial result ({\cite[Proposition 8]{igdl}}) proved under the  hypothesis $\h(H,V)=0.$

\begin{prop}\label{matrici}
Suppose that $h_1,\dots,h_d$ invariably generate $H.$
 Let $w_1,\dots,w_d\in V^u$ with
$w_i=(w_{i,1},\dots,w_{i,u})$. For $j\in \{1,\dots,u\}$, consider the vectors
$$r_j=\big(\pi_j(w_1), \dots , \pi_j(w_d)\big)=(w_{1,j}, \dots, w_{d,j})\in V^{d}.$$
Then $h_1w_1, h_2w_2,\dots,h_d w_d$ invariably generate
 $V^u\rtimes H$ if and only if the vectors $r_1,\dots,r_u$ are linearly independent modulo $D+W$ where
 $$\begin{aligned}D&=\{ \big(\delta(h_1), \dots , \delta(h_d)\big)\in V^d \mid \delta\in \der(H,V)\},\\ W&=\{(u_1,\dots,u_d)\in V^{d}\mid  u_i \in [h_i,V], \ i=1, \dots , d\}.\end{aligned}$$
In particular, there exist  elements $w_1,\dots,w_d\in V^u$ such that $h_1w_1, h_2w_2,\dots,h_d w_d$ invariably generate
  $V^u\rtimes H$ if and only if
$u\leq nd-\dim_F(D+W).$
\end{prop}

\begin{proof}
Let $g_i=y_ix_i$
 with $x_i\in H$ and $y_i=(y_{i,1},\dots,y_{i,u})\in V^u$ for $i\in\{1,\dots,d\}$
and let $X_{g_1,\dots, g_d}=\langle (h_1w_1)^{g_1},\dots, (h_dw_d)^{g_d}\rangle.$
We have
$$(h_iw_i)^{g_i}=(h_i^{y_i}w_i)^{x_i}=
h_i^{x_i}([h_i,y_i]+w_i)^{x_i}=h_i^{x_i}z_i$$
where
 $z_i=([h_i,y_i]+w_i)^{x_i}\in V^u$.
 Then $X_{g_1,\dots, g_d}=G$ if and only if the vectors 
 $$\big(\pi_j(z_1), \dots , \pi_j(z_d)\big)=
\big(   ([h_1,y_{1,j}]+w_{1,j})^{x_1}, \dots , ([h_d,y_{d,j}]+w_{d,j})^{x_d}\big) \in V^d,
 $$
 for $j\in \{1,\dots,u\}$,  
are linearly independent modulo the subspace
$$\begin{aligned}
{D}^*&=\{ \big(\delta(h_1^{x_1}), \dots , \delta(h_d^{x_d})\big)\in V^d \mid \delta\in \der(H,V)\}\\
&= \left\{\left(\left(\delta(h_1)-[h_1,\delta(x_1^{-1})]\right)^{x_1}\!\!\!\!\!,\dots , \left(\delta(h_d)-[h_d,\delta(x_d^{-1})]\right) ^{x_d} \right)\!\in\! V^d \mid \delta\in \der(H,V)\right\}
\end{aligned}$$
(we have indeed that $\delta(h^x)=\delta(x^{-1}hx)=\delta(x^{-1}h)^x+\delta(x)=
(\delta(x^{-1}h)+\delta(x)^{x^{-1}})^x=
(\delta(x^{-1})^h+\delta(h)-\delta(x^{-1}))^x)=(\delta(h)-[h,\delta(x^{-1})])^{x}$).

Note that the map $f_{(x_1, \dots, x_d)}:V^d \mapsto V^d$ defined by
 $$f_{(x_1, \dots, x_d)}(v_1, \dots , v_d)=(v_1^{x_1}, \dots , v_d^{x_d})$$ is an isomorphism.
Therefore $X_{g_1,\dots, g_d}=G$ if and only if  the vectors 
 $$\big(   [h_1,y_{1,j}]+w_{1,j}, \dots , [h_d,y_{d,j}]+w_{d,j}\big)=r_j+\big( [h_1,y_{1,j}],  \dots , [h_d,y_{d,j}]\big),
 $$
 for $j=1,\dots,u,$
 are linearly independent modulo
  the subspace
$$\left\{\left(\left(\delta(h_1)-[h_1,\delta(x_1^{-1})]\right) , \dots , \left(\delta(h_d)-[h_d,\delta(x_d^{-1})]\right) \right)\in V^d \mid \delta\in \der(H,V)\right\}.$$
Since this condition has to hold for every choice of  $y_i \in V^u$ and $x_j \in H$,  this means that
 the elements 
 ${r_1}, \dots , {r_u}$ have to be
linearly independent modulo
  the subspace
$D+W,$ as required.
\end{proof}

\begin{lemma}\label{dimen}In the situation described  in Proposition \ref{matrici}, and using the same notations, 
	we have that $$nd-\dim_F(D+W)\geq \!\!
	\sum_{1\leq i\leq d}\!\dim_FC_V(h_i)-m,$$
	with $m=\dim_F \h(H,V).$
\end{lemma}

\begin{proof}
	Firstly, notice that $$\dim_F W=\sum_{1\leq i \leq d}\dim_F[h_i,V]=\sum_{1\leq i \leq d}\left(n-\dim_FC_V(h_i)\right)=nd-\sum_{1\leq i\leq d}\dim_FC_V(h_i).$$
	Moreover $D\cap W$ contains
	$I=\{ \big(\delta(h_1), \dots , \delta(h_d)\big)\in V^d \mid \delta\in \ider(H,V)\},$ which is $F$-isomorphic to $\ider(H,V),$
	and consequently
	$$\begin{aligned}\dim_F(D+W)-\dim_F(W)&=\dim_F((D+W)/W)=\dim_F(D/(D\cap W))\\&\leq\dim_F D/I = \dim_F (\der(H,V)/\ider(H,V))\\&=\dim_F \h(H,V)=m.\end{aligned}$$
We conclude
	$$\begin{aligned}\dim_F(D+W)&\leq \dim_FW+\dim_F \h(H,V)\\ &\leq nd-\sum_{1\leq i\leq d}\dim_FC_V(h_i)+m.\qedhere
	\end{aligned}$$
\end{proof}

\section{First cohomology groups for finite groups}

For all  this section we will assume that $H$ is a finite group, $F$ is a field of finite characteristic and  $V$ is a faithful and absolutely irreducible $FH$-module. Moreover  let $n=\dim_FV,$  $m=\dim_{F} \h(H,V)$. 

\

In the proof of our main result we will need a good bound upper bound for $m.$
The following result is available (see \cite[Theorem A]{AG}, \cite[Theorem 1]{guho}):

\begin{prop}$m\leq n/2\leq n-1.$
\end{prop}

Guralnick  made a conjecture that there should be a universal bound on the dimension
of the first cohomology groups $\h(H, V ),$ where $H$ is a finite group and $V$ is an absolutely irreducible faithful representation for $H.$ The conjecture reduces to the case where $H$ is a finite simple group. Very recently, computer calculations of Frank L\"ubeck, complemented by those of Leonard
Scott and Tim Sprowl, have provided strong evidence that the Guralnick conjecture may
unfortunately be false. For our purpose is not necessary that the Guralnick conjecture is true. A much weaker version, which will be discussed in this section,
is enough. First we need a preliminary lemma.

\begin{lemma}\label{diff}
If $m\neq 0,$ then:
\begin{enumerate}
\item $H$ has a unique minimal normal subgroup $N$ and $N$ is nonabelian.
\item If $S$ is a component of $N$ and $W$ is an irreducible $FN$-submodule of $V$ which is not centralized by $S,$ then  the other components of $N$ act trivially on $W$. 
\item $m\leq \dim_F \h(S,W)$ for any irreducible submodule of $V$ which is not centralized by $S.$
\item Every element of $C_H(S)$ fixes at least a nonzero vector of $V.$

\end{enumerate}
\end{lemma}
\begin{proof}
It is well known that if $K$ is an extension field of $F,$ then $\h(H,V)\otimes_FK$ and $\h(H,V\otimes_FK)$ are naturally isomorphic, so may assume that $F$ is algebraically closed.
The first three statements are proved in \cite[Lemma 5.2]{GKKL}. Let $\Omega$ be the 
set of irreducible $FN$-submodules of $V$ which are not centralized by $S$ and let
$U=\sum_{W\in \Omega}W.$ Let $I$ be the stabilizer of $U$ in $H$. It follows from (2) that $I = N_H(S).$ Since $V$ is irreducible, $U$ is an irreducible $I$-module. Let $R =SC_H(S).$ By \cite[Lemma 3.4]{GKKL}, $\h(H,V) =
\h(I,U)$ and, by \cite[Lemma 3.11]{GKKL}, $\dim \h(I,U) \leq
\dim \h(R,U).$ Since $R = S \times C_H(S),$ $U$ is a direct sum of modules of the form
$W \otimes X$ where $W\in \Omega$ and each $X$ is an irreducible $C_H(S)$-module. By \cite[Lemma 3.10]{GKKL} if all the $X$ are nontrivial  $C_H(S)$-modules then
$\h(R,U) = 0,$ and so $\h(H,V) = 0.$ So $C_H(S)$ acts trivially on some of the direct factors of $U.$
\end{proof}

\begin{prop}\label{como} Denote by $p$ the probability that an element $h$ of $H$ centralizes a non-zero vector of $V.$ There exists a constant $\alpha$ (independent on the choice of $H$ and $V$) with the property that if $|V|\leq |H|,$ then either $m\leq \alpha$ or $p|H|\geq  m^2.$
\end{prop}

\begin{proof}
We may assume $m\neq 0.$ By Lemma \ref{diff}, $H$ has a unique minimal normal subgroup $N\cong S^t$ where $S$ is a nonabelian simple group. First assume $t\neq 1$. We may identify $H$ with a subgroup
of $\aut S \wr K$ being $K$ the transitive subgroup of $\perm(t)$ induced by the conjugacy action of $H$ on the components.
 It follows from Lemma \ref{diff} (3), that
$$p|H|\geq |C_H(S)|\geq \frac{|H|}{t|\aut S|}\geq \frac {|K||S|^{t-1}}{t|\out S|},$$ while, since $2^n\leq q^n\leq |H|,$ we have
$$m<n \leq \log |H| \leq \log (|\aut S|^{t}|K|)\leq \log(|S|^{2t}|K|) .$$ It follows
that there exists $\tau$ such that $p|H|\geq m^2$ if $|S|\geq \tau.$ On the other hand, there are only finitely many possible pairs $(S,W)$ where $S$ is a simple group of order at most $\tau$ and $W$ is an irreducible $FS$-module with
$\h(S,W)\neq 0$ (since $\h(S,W)=0$ if $S$ and $W$ have coprime orders) so it follows from Lemma \ref{diff} (3) that there exists $\alpha$ such that $m\leq \alpha$ whenever $|S|\leq \tau.$

So we may assume that $H$ is an almost simple group, and that $S=\soc H$ is a finite group of Lie type or alternating group, since the number of possibilities for $H$ and
$V$  when $H$ is sporadic and $\h(H, V) \neq 0$ is finite. Let $r$ be the characteristic of $F.$ The condition $m\neq 0$ implies that $r$ divides $|H|.$ Moreover all the elements of a Sylow $r$-subgroup of $H$ centralize at least a non-zero vector of $V,$ so $p|H|\geq |H|_r,$ the largest power of $r$ dividing $|H|.$
We have three possibilities:

a) $S=\alt(k).$ Since $2^n\leq q^n\leq |H|\leq k!,$ we have $n\leq k\log k.$
By \cite[Corollary 3]{gk}, we have $m\leq n/(f-1)$ being $f$ the largest prime such that $f\leq k-2.$ Nagura \cite{na} proved that for each $x \geq 25,$ the interval 
$[x,6x/5]$ contains a prime, hence if $k$ is large enough then $(f-1)\geq k/2$ and consequently $m\leq k\log k /(f-1) \leq 2\log k.$ We cannot have $r>k/2,$
otherwise a Sylow $r$-subgroup of $H$ would be cyclic and this would implies $m=0$ (see \cite[Proposition 3.4]{gsg}). But then $k=ar+b$ with $a,b \in \mathbb N, a\geq 1$
and $b<r\leq k/2.$ So $(k!)_r\geq r^a\geq r\cdot a\geq k/2.$ We conclude that $|H|_r\geq k/2\geq (2\log k)^2\geq m^2$ if $k$ is large enough, say $k\geq \tau.$ Since there are only finitely many possibilities of $k\leq \tau$ and an absolutely irreducible $\alt(k)$-module $V$ such that $\h(\alt(k),V)\neq 0,$ we are done in this case.

b) $S$ is a group of Lie type defined over a field whose characteristic is different
from the characteristic $r$ of $F.$ Let us denote by $\delta(S)$ the smallest degree of a nontrivial irreducible representation of $S$ in cross characteristic. Lower bounds for the degree of irreducible representations of finite groups of Lie
	type in cross characteristic were found by Landazuri and Seitz \cite{ventidue} and improved
	later by Seitz and Zalesskii \cite{sz} and Tiep  \cite{diciotto}. It turns out that $\delta(S)$ is quite large, and, apart from finitely many exceptions, we have $r^{\delta(S)} >|\aut S|,$ in contradiction with $r^{\delta(S)}\leq |V| < |H|\leq |\aut S|.$

c)  $S$ is a group of Lie type defined over a field whose characteristic coincides with
the characteristic $r$ of $F.$ We have $p|H|\geq |H|_r\geq |S|^{1/3}$ (see \cite[Proposition 3.5]{artin}). On the other hand $|V|\leq |H|\leq |S|^2,$ hence
$m \leq n/2 \leq \log |S|$ and again we can conclude that $p|H|\geq |S|^{1/3} \geq  \log^2 |S|\geq m^2$ if $|S|$ is large enough.
\end{proof}

\section{Auxiliary results}

We begin this section with an elementary result in probability theory, which will play a crucial
role in our considerations. Let us denote by $B(m,p)$ the binomial random variable of parameters $m$ and $p.$

\begin{prop}\label{binomiale}
	For every real number $0<\epsilon <1,$ there exists an absolute constant $\gamma_\epsilon$ such that,
	for any positive integer $l$ and any positive real number $p<1,$ we have that $P(B(m,p)\geq l)\geq \epsilon$ if $m\geq \gamma_\epsilon l/p.$
\end{prop}

\begin{proof}
	Let $M(t)$ be the moment generating function of the
	random variable
	$X=B(m,p).$ We have $M(t)=(pe^t+(1-p))^m.$
	By Chernoff's bounds (see for example \cite[Chapter 8, Proposition 5.2]{ross}),
	$P(X\leq a)\leq e^{-ta}M(t)$ for every real negative number $t$.
	Taking $t=-1$ and $a=l$, we deduce $$P(X\leq l)\leq e^l(1-\alpha p)^m \text { with } \alpha=(1-1/e).$$
	In particular $P(X\geq l)\geq 1-e^l(1-\alpha p)^m,$ and
	we reduce to prove that there exists $\gamma_\epsilon$ such that
	$e^l(1-\alpha p)^m \leq (1-\epsilon)$ if $m\geq \gamma_\epsilon l/p.$
	It suffices to choose $\gamma_\epsilon$ such that $(1-\alpha p)^{\gamma_\epsilon/p}\leq (1-\epsilon)/e.$ Since $(1-\alpha p)^{\gamma_e/p}=(1-\alpha p)^{\alpha\gamma_\epsilon/\alpha p}\leq e^{-\gamma_\epsilon\alpha}$, it suffices to take
	$\gamma_\epsilon \geq (1-\log(1-\epsilon))/{\alpha}.$
\end{proof}

From now on we will use the notation $\langle x_1,\dots,x_d\rangle_I=G$ to say that
$x_1,\dots,x_d$ invariably generate $G.$

\begin{lemma}\label{relativo}Assume that $G$ is a finite  group with trivial Frattini subgroup and let $I=I_G(A), R=R_G(A), U$ be as in the statement of Lemma \ref{corona}. Let $g_1,\dots,g_t \in G.$
	If $\langle g_1U,\dots,g_tU\rangle_I=G/U$ and
	$\langle g_1R,\dots,g_tR\rangle_I=G/R,$ then
	$\langle g_1,\dots,g_t\rangle_I=G.$
\end{lemma}

\begin{proof}
	Let $x_1,\dots,x_t\in G$ and consider $K=\langle g_1^{x_1},\dots,g_t^{x_t}\rangle.$	Since $\langle g_1U,\dots,{g_t}U\rangle_I=G/U$ (and resp. $\langle g_1R,\dots,g_tR\rangle_I=G/R$ ) we have $KU=G$ (and resp.
	$KR=G$). But then $K=G$ by Lemma \ref{sotto}.
\end{proof}

\begin{lemma}\label{nonabcr}\cite[Proof of Theorem 4.1]{ig}
Denote by $P_G^*(k)$ the probability that $k$ randomly chosen elements $g_1,\dots,g_k\in G$ have the property that there exists a maximal subgroup $M$ of $G$ such that the primitive group $G/\core_G(M)$ is
not of affine type and $g_1,\dots,g_k \in \cup_{g\in G}M^g.$
For any $\epsilon>0,$ there exists $c_{\epsilon}$ such that
$P^*_G(k)\leq \epsilon$ for any finite group $G$ and any $k\geq
c_\epsilon(\log|G|)^2.$
\end{lemma}

\begin{proof}
This result is part of the proof of \cite[Theorem 4.1]{ig}. In the first part of that proof, the authors show that $$P_G^*(k)\leq c_1\sqrt{|G|^3}\left(1-c_2/\log|G|\right)^k$$
for some absolute constants $c_1$ and $c_2$ and notice that there exists $c_3$ such that if  $k\geq c_3(\log|G|)^2,$ then
the right-hand tends to zero as $|G|\to \infty.$ 
\end{proof}

The authors of \cite{ig} notice that the proof of the previous result uses \cite[Theorem 8.1]{fg1}, which in turn relies on the conjecture, due to Boston and Shalev, stating that there exists an absolute constant $\epsilon > 0$ such that the proportion of fixed-point-free elements in any finite simple
transitive permutation group is at least $\epsilon.$ However,  in \cite{ig} it is noticed that a
weaker version of \cite[Theorem 8.1]{fg1} allows to prove
that for any $\epsilon>0$ there exists $c_{\epsilon}$ such that
$P^*_G(k)\leq \epsilon$ for any finite group $G$ and any
$k\geq c_\epsilon(\log|G|)^3|G|^{1/3}.$ This weaker version of Lemma \ref{nonabcr} still suffices for our purpose.

\

We now introduce some other definitions. Let $N$ be a normal subgroup of a finite group $G$ and let $\Lambda_{G,N}$ be the set of the  ordered sequence $(x_1,\dots,x_d)\in G^d$ (for any possible choice of $d)$ having the property that $\langle Nx_1,\dots Nx_d\rangle_I=G/N.$
For  $\xi=(x_1,\dots,x_d)\in \Lambda_{G,N},$
denote by $P_I(G,N,\xi,k)$ the probability that $k$ randomly chosen elements $y_1,\dots,y_k$ of $G$ have the property that
$\langle x_1,\dots,x_d,y_1,\dots,y_k\rangle_I=G$ and
let $$P_I(G,N,k)=\inf_{\xi \in \Lambda_{G,N}}P_I(G,N,\xi,k).$$
We have in particular
$$P_I(G,k_1+k_2)\geq P_I(G/N,k_1)P_I(G,N,k_2)$$ for every
$k_1,k_2\in\mathbb N.$

\begin{lemma}\label{cruciale} Assume that $G$ is a finite  group with trivial Frattini subgroup and let $I=I_G(A),$ $R=R_G(A),$ $U$ be as in the statement of Lemma \ref{corona}. 
	There exists an absolute constant
	$c,$ independent on the choice of $G,$ such that if $k\geq c\sqrt{|G|},$ then
	$P_I(G,U,k)\geq {3}/{4}.$
\end{lemma}

\begin{proof} 
	It suffices to prove that there exists an absolute constant
	$c,$ independent on the choice of $G$ and $\xi,$ such that if $k\geq c\sqrt{|G|},$ then $P_I(G,U,\xi,k)\geq {3}/{4}$
	for every  $\xi\in \Lambda_{G,U}.$ So we fix  $\xi=(x_1,\dots,x_d)\in \Lambda_{G,U}$ and we estimate $P_I(G,U,\xi,k).$
	Let $\bar G=G/R$ and $\bar \xi=(x_1R,\dots,x_dR) \in \bar G^d.$
	By Lemma
 \ref{relativo}, given $(y_1,\dots,y_k)\in H^k,$ if $\langle x_1R,\dots,x_d,y_1R,\dots,y_kR\rangle_I=\bar G$ then
	$\langle x_1,\dots,x_d,y_1,\dots,y_k\rangle_I=G,$ hence
	$P_I(G,U,\xi,k)\geq P_I(\bar G,\bar U,\bar \xi,k),$ and so we may assume $R=1.$ We have $R=R_G(A)$ where $A$ is an irreducible $G$-group: in particular $G=L_\delta$ where
	$L$ is the monolithic primitive group associated to $A$ and
	 $\delta=\delta_G(A).$ 
	 
	 First assume that $A$ is nonabelian. We want to count the $k$-tuples
 $(y_1,\dots,y_k)$ such that $\langle x_1,\dots,x_d,y_1,\dots,y_k\rangle_I=G.$
	  If $\langle x_1,\dots,x_d,y_1,\dots,y_k\rangle_I\neq G,$ then
	 there exists a maximal subgroup $M$ of $G$ such that
	$$\{x_1,\dots,x_d,y_1,\dots,y_k\}\subseteq \cup_{g\in G}M^g.$$
	 This $M$ cannot contain $U$, otherwise
	 $\{Ux_1,\dots,Ux_d\}\subseteq \cup_{gU\in G/U}(M/U)^{gU},$
	 against the property that $Ux_1,\dots Ux_d$ invariably generate $G/U.$ Thus $MU=G$ and, consequently, being $U\cong A^\delta$ with $A$ nonabelian, the primitive group $G/\core_G(M)$ is not of affine type and  $\{y_1,\dots,y_k\}\subseteq \cup_{g\in G}M^g.$ Hence, by Lemma \ref{nonabcr},
	$P_I(G,U,\xi,k)\geq 1-P_G^*(k)\geq 3/4$ if $k\geq c_{1/4}(\log|G|)^2.$ Clearly there exists an absolute constant $c^*$ such that $c_{1/4}(\log m)^2\leq c^*\sqrt m$
	for every $m\in \mathbb N.$
	
	We assume now that $A$ is abelian. In this case $A$ is $G$-isomorphic to an irreducible $G$-module $V.$ Moreover
	either $V\cong C_p$ is a trivial $G$-module and $G\cong (C_p)^\delta$  or $G\cong U\rtimes H$
	where $H$ acts in the same say on each of the $\delta$ factors of $U\cong V^\delta$ and this action
	is faithful and irreducible.
	
	In the first case, denoting by $P(C_p^\delta,k)$ the probability that $k$ elements of $C_p^\delta$ generate $C_p^\delta$, we have
	$$P_I(G,U,\xi,k)\geq  P_I(C_p^\delta,k)=P(C_p^\delta,k)=\prod_{k-\delta+1\leq i \leq k}\left(1-\frac{1}{p^i}\right)\geq 1-\frac{p^\delta-1}{p-1}\frac{1}{p^k}\geq 1-\frac{p^\delta}{p^k},$$
	in particular $P_I(G,U,\xi,k)\geq 3/4$ if $k\geq \delta+2:$ it suffices to choose $c\geq 3/\sqrt 2,$ since in that case
	$c\sqrt{|G|}\geq 3p^{\delta/2}/\sqrt 2\geq \delta+2.$
	
	In the second case, we have $G=V^\delta\rtimes H$ and we estimate $P_I(G,U,\xi,k)$ by applying Proposition \ref{matrici}.
	Let $F=\End_HV,$ with $|F|=q,$ and let $n=\dim_FV$ (so in particular $|V|=q^n$). For $i\in\{1,\dots,d\},$ let $x_i=k_iw_i$ with $w_i\in V^\delta$ and $k_i\in H.$ Now choose
	$y_1,\dots,y_k\in G$, where $y_j=h_jw_j^*$ with $w_j^*\in V^\delta$
	and $h_j\in H.$ Given a subset $J=\{j_1,\dots,j_f\}$ of $I=\{1,\dots,k\}$, 
	consider the projection $\pi_J:V^{d+k}\to V^f$ defined by setting
	$\pi_J(v_1,\dots,v_d,v_1^*,\dots,v_k^*)=(v_{j_1}^*,\dots,v_{j_f}^*)$ and for $t\in \{1,\dots,\delta\}$ let
	$$\begin{aligned}r_t&=(\pi_t(w_1),\dots,\pi_t(w_d),\pi_t(w_1^*),\dots,\pi_t(w_k^*))\in V^{d+k},\\ r_{t,J}&=\pi_J(r_t)=(\pi_t(w_{j_1}^*),\dots,\pi_t(w_{j_f}^*))\in V^f.
	\end{aligned}$$
	Moreover let
	$$
	\begin{aligned}W&=\{(u_1,\dots,u_d,u_1^*,\dots,u_k^*)\!\mid \! u_i\!\in [k_i,V] \text { for }1\leq i\leq d, u_j^*\!\in [h_j,V]\text { for }1\leq j\leq k\},\\
D&=\{ \big(\delta(k_1), \dots , \delta(k_d),\delta(h_1), \dots , \delta(h_k)\big)\in V^{d+k} \mid \delta\in \der(H,V)\},\\
	W_J&=\pi_J(W)=\{(u_{j_1}^*,\dots,u_{j_f}^*)\mid u_{j_i}^*\in [h_{j_i},V]\text { for }1\leq i\leq f\},\\
D_J&=\pi_J(D)=\{ \big(\delta(h_{j_1}), \dots , \delta(h_{j_f})\big)\in V^{f} \mid \delta\in \der(H,V)\}.
	\end{aligned}$$
Notice that	 if the vectors $r_{1,J},\dots,r_{\delta,J}$ are $F$-linearly independent modulo $W_J+D_J$ for some $J\subseteq I, $ then $r_1,\dots,r_\delta$ are linearly independent modulo $W+D$ and, by Proposition \ref{matrici},
	 $\langle x_1,\dots,x_d,y_1,\dots,y_k\rangle_I=G.$ Now let $m=\dim_F \h(H,V)$ and distinguish the following cases:

 a) $|H|\geq |V|m^2$.
Let $\Delta_{l}$ be the subset of $H^k$ consisting of the $k$-tuples $(h_1,\dots,h_k)$ with the property that $C_V(h_i)\neq 0$ for at least $l$ different choices of $i\in\{1,\dots,k\}.$  If $(h_1,\dots,h_k)\in \Delta_{l},$ then, by Lemma \ref{dimen}, $W_I+D_I$ is a subspace of $V^k\cong F^{nk}$ of codimension at least $l-m$: so the probability that $r_{1,I},\dots,r_{\delta,I}$ are $F$-linearly independent modulo $W_I+D_I$ is at least
$$\begin{aligned}p_{l}&=\left(\frac{q^{nk}-q^{nk-l+m}}{q^{nk}}\right)
\cdots\left(\frac{q^{nk}-q^{nk-l+m+\delta-1}}{q^{nk}}\right)\\
&=\left(1-\frac{1}{q^{l-m}}\right)\dots \left(1-\frac{q^{\delta-1}}{q^{l-m}}\right)\geq 1-\left(\frac{q^\delta-1}{q-1}\right)\frac{1}{q^{l-m}}.
\end{aligned}$$
Notice in particular that $p_{l}\geq 7/8$ if $l\geq \delta+m+3$
hence
$$P_I(G,U,\xi,k)\geq \frac{7\rho}{8}$$
where $\rho$ denotes the probability that $(h_1,\dots,h_k)\in 
\Delta_{\delta+m+3}.$ Therefore in order to conclude our proof it suffices to show that there exists a constant $c_1$ such that $\rho\geq 6/7$ if $k\geq c_1\sqrt{|G|}.$ Let $p$ be the probability that a randomly chosen element $h$ of $H$ satisfies the condition $C_V(h)\neq 0.$ We have
$$\rho=P(B(k,p)\geq \delta+m+3).$$
Therefore, by Proposition \ref{binomiale}, $\rho\geq 6/7$
if $k\geq \gamma(\delta+m+3)/p,$ being $\gamma=\gamma_{6/7}.$
Let $v$ be a fixed nonzero vector of $V$ and let
$H_v$ be the stabilizer of $v \in H.$ Clearly $p\geq |H_v|/|H|\geq 1/|V|= 1/q^n,$ hence $\rho\geq 6/7$ if
$k\geq \gamma (\delta+m+3)q^n.$
Since we are assuming $|G|=|H||V|^\delta\geq q^nm^2q^{n\delta}=q^{n(\delta+1)}m^2,$ there exists an absolute constant $c_1$ such that $\gamma(\delta+m+3)q^n\leq c_1mq^{n(\delta+1)/2}\leq c_1\sqrt{|G|}.$ Hence
$\rho\geq 6/7$ if $k\geq c_1\sqrt{|G|}.$

b) $|H|\geq |V|$ and $m\leq \alpha$, where $\alpha$ is the constant with appears in the statement of Proposition \ref{como}. Arguing as before, we have that $P_I(G,U,\xi,k)\geq 3/4$ if 
$$\gamma (\delta+m+3)q^n\leq  \gamma (\delta+\alpha+3)q^n\leq k.$$ We are assuming 
$|G| =  |H||V|^\delta\geq q^nq^{n\delta}=q^{n(\delta+1)}$, so there exists a constant $c_2$ such that $\gamma(\delta+m+3)q^n\leq \gamma(\delta+\alpha+3)q^n\leq c_2q^{n(\delta+1)/2}\leq c_2\sqrt{|G|}.$

c) $|V|\leq |H|\leq |V|m^2$ and $m >\alpha.$ We repeat the same argument as above, using  the bound $p\geq |H|/m^2,$ ensured by Proposition \ref{como}. We find that $P_I(G,U,\xi,k)\geq 3/4$ if
$k\geq \gamma(\delta+m+3)|H|/m^2.$ Since $|H|^{1/2}\leq q^{n/2}m,$ there exists 
 a constant $c_3$ such that 
 $$\frac{\gamma(\delta+m+3)|H|}{m^2} \leq \frac{\gamma(\delta+4)|H|}{m} \leq \gamma(\delta+4) |H|^{1/2}q^{1/2}\leq c_3 |H|^{1/2}q^{\delta/2}\leq c_3\sqrt{|G|}.$$
 
d)  $|H|\leq |V|=q^n.$
	Let $\Omega_{l}$ be the subset of $H^k$ consisting of the $k$-tuples $(h_1,\dots,h_k)$ with the property that $h_i=1$ for at least $l$ different choices of $i\in I=\{1,\dots,k\}.$ For a given $\omega \in \Omega_l,$  let $J_\omega=\{i\in I\mid h_i=1\}$ and let $l_\omega=|J_\omega|\geq l.$ We have that ${W_{J_\omega}}+{D_{J_\omega}}=0,$
	so the probability that ${r_{1,J_\omega}},\dots,{r_{\delta,J_\omega}}$ are $F$-linearly independent modulo ${W_{J_\omega}}+{D_{J_\omega}}=0$ is at least
	$$\begin{aligned}q_{\omega}&=\left(\frac{q^{n{l_\omega}}-1}{q^{n{l_\omega}}}\right)
	\cdots\left(\frac{q^{n{l_\omega}}-q^{n{l_\omega}-\delta-1}}{q^{n{l_\omega}}}\right)\\
	&=\left(1-\frac{1}{q^{n{l_\omega}}}\right)\dots \left(1-\frac{q^{\delta-1}}{q^{n{l_\omega}}}\right)\geq 1-\left(\frac{q^\delta-1}{q-1}\right)\frac{1}{q^{nl_\omega}}
	\geq 1-\left(\frac{q^\delta-1}{q-1}\right)\frac{1}{q^{nl}}.
	\end{aligned}$$
	Notice in particular that $q_{\omega}\geq 7/8$ if $nl\geq \delta+3$
	hence
	$$P_I(G,U,\xi,k)\geq \frac{7\rho}{8}$$
	where $\rho$ denotes the probability that the number of trivial entries in $(h_1,\dots,h_k)$ is larger than $\lceil(\delta+3)/n\rceil\leq \delta+3.$ Therefore in order to conclude our proof it suffices to show that there exist a constant $c_4$ such that $\rho\geq 6/7$ if $k\geq c_4\sqrt{|G|}.$ 
	By Proposition \ref{binomiale},
	$\rho \geq 6/7$ if
	$k\geq \gamma (\delta+3)|H|,$ being $\gamma=\gamma_{6/7}.$
	Since
	$|G|=|H||V^\delta|$ and $|H|\leq |V|,$ there
	there exists an absolute constant $c_4$ such that
	$${\gamma(\delta+3)|H|}\leq {c_4|H|q^{n/2(\delta-1)}}\leq c_4|H|^{1/2}q^{n\delta/2}\leq c_4\sqrt{|G|}.$$
	Hence
	$\rho\geq 6/7$ if $k\geq c_4\sqrt{|G|}.$
	
If we take $c=\max\{c^*,\sqrt{3}/2,c_1,c_2,c_3,c_4\},$ we have	$P_I(G,U,k)\geq {3}/{4}$.
\end{proof}

\section{Proof of Theorem \ref{due}}

An easy argument (see the end of this section) shows that in order to prove Theorem \ref{due}
it suffices to prove the statement for a particular choice of
the positive real number $\epsilon.$ So the proof of Theorem \ref{due} will be a corollary of the following result:

\begin{thm}\label{duenoni}
	Let $\bar c=15c$ where c is constant introduced in the statement of Lemma \ref{cruciale}. If $G$ is a finite  group and $k\geq \bar c\sqrt{|G|},$ then $P_I(G,k)\geq 2/9.$
\end{thm}
\begin{proof}
	Let $F_1=\frat(G).$ By Lemma \ref{corona}, there exists a crown
	$I_1/R_1$ of $G$ and a nontrivial normal subgroup $U_1/F_1$ of
	$G/F_1$ such that $I_1/F_1=R_1/F_1\times U_1/F_1.$ If $U_1=G,$
	then, since $k \geq \bar c\sqrt{|G|} \geq c\sqrt{|G|},$ $P_I(G,k)=P_I(G/F_1,k)\geq 3/4$ by Lemma \ref{cruciale}.
	Otherwise let $F_2/U_1=\frat(G/U_1):$ again by  Lemma \ref{corona}, there exists a crown
	$I_2/R_2$ of $G$ and a nontrivial normal subgroup $U_2/F_2$ of
	$G/F_2$ such that $I_2/F_2=R_2/F_2\times U_2/F_2.$ If $U_2=G$
	then there exists two integers $k_1$ and $k_2$, both larger than $c\sqrt{|G|}$ and such that
	$k_1+k_2\leq \bar c\sqrt{|G|}.$ 
	By Lemma \ref{cruciale}, we have 
	$$\begin{aligned}P_I(G,k)&\geq P_I(G,k_1+k_2)\geq P_I(G/U_1,k_1)P(G,U_1,k_2)\\&=P_I(G/F_2,k_1)P(G,U_1,k_2)\geq \left( \frac {3}{4} \right)^2.\end{aligned}$$ Finally assume $G\neq U_2$. We have
	that $U_2/F_2\sim_G A_2^{\delta_2}$ and $U_1/F_1\sim_G A_1^{\delta_1}$
	where $A_1$ and $A_2$ are non $G$-equivalent chief factors of $G$: in particular $|A_1||A_2|\geq 6$ and consequently
	$|G|/|U_2|\leq |G|/6.$ But then
	$$\begin{aligned}k&\geq \bar c\sqrt{|G|} = 15c\sqrt{|G|}\geq  30\cdot c\sqrt{\frac{|G|}6}+c\sqrt{\frac{|G|}2}+c\sqrt{|G|}+4\\
	&\geq 2\left \lceil \bar c\sqrt{|G/U_2|}\right\rceil + \left \lceil c\sqrt{|{G}/{U_1}|}\right\rceil+\left \lceil c\sqrt{|G|}\right\rceil
	\end{aligned}$$
	and there exist three integers $k_1,$ $k_2$ and $k_3$ such that
	$$k_1+k_2+k_3\leq k,\quad k_1\geq 2\left \lceil \bar c\sqrt{|G/U_2|}\right\rceil,\quad  k_2 \geq c\sqrt{|G/U_1|} \text { and } k_3 \geq c\sqrt{|G|}.$$ By induction, if $t\geq \bar c\sqrt{|G/U_2|},$ then $p=P_I(G/U_2,t)\geq 2/9$ and consequently $$P_I(G/U_2,2t)\geq 1-(1-p)^2=2p-p^2\geq 32/81.$$
	Hence the probability that $(x_1,\dots,x_{k_1})\in G^{k_1}$
	satisfies the condition $$\langle x_1U_2,\dots,x_{k_1}U_2\rangle_I=G/U_2$$
	is at least $32/81.$
	Applying twice Lemma \ref{cruciale}, we conclude
	that $$P_I(G,k_1+k_2+k_3)\geq \frac{32}{81}\cdot \frac{3}{4}\cdot \frac{3}{4}=\frac{2}{9}.\qedhere$$
\end{proof}

\begin{proof}[Proof of Theorem \ref{due}]
	Given $0 < \epsilon < 1,$ there exists a positive integer $t$ such that
	$\epsilon \geq (7/9)^t.$ Let $\tau_\epsilon=t(1+\bar c)$ where $\bar c$ is the constant introduced in the statement of Theorem \ref{duenoni}. Let $k$ be an integer
	larger than $\tau_\epsilon \sqrt{|G|}.$ We have
	$$t\left \lceil \bar c\sqrt{|G|}\right\rceil \leq t\bar c\sqrt{|G|}+t=\tau_{\epsilon}{\sqrt {G}}\leq k$$
	hence there exist $t$ integers $k_1,\dots,k_t$
	such that $k_1+\cdots+k_t \leq k$ and $k_i \geq \bar c\sqrt{|G|}$ for all $i\in \{1,\dots,t\}.$ It follows
	$$P_I(G,k)\geq P_I(G,k_1+\dots+k_t)\geq 1-\prod_{1\leq i\leq t}\left(1-P_I(G,k_i)\right)\geq  1-(7/9)^t\geq 1 - \epsilon $$
	since $P_I(G,k_i)\geq 2/9$ by Theorem \ref{duenoni}.
\end{proof}


\end{document}